\documentclass{amsart2010}
\usepackage{graphicx}
\usepackage{amsmath}
\usepackage{amsfonts}
\usepackage{amssymb}

\newtheorem{thm}{Theorem}

\newtheorem{lem}[thm]{Lemma}
\newtheorem{prop}[thm]{Proposition}
\theoremstyle{definition}

\newtheorem*{theorem*}{Theorem}

\theoremstyle{remark}

\numberwithin{equation}{section}

\def\<{\langle}
\def\>{\rangle}
\begin{document}
	\title[]{Universal Toeplitz operators on the Hardy space over the polydisk}
	\author{Marcos Ferreira, S. Waleed Noor}%
	\address{Departamento de Ciências Exatas e Tecnológicas, UESC, Ilhéus, Brasil.}
\email{$\mathrm{msferreira@uesc.br}$ (1st author).}
	
	\address{IMECC, Universidade Estadual de Campinas, Campinas-SP, Brazil.}
	\email{$\mathrm{waleed@unicamp.br}$ (2nd author).}
	\begin{abstract}  The \emph{Invariant Subspace Problem} (ISP) for Hilbert spaces asks if every bounded linear operator has a non-trivial closed invariant subspace. Due to the existence of \emph{universal} operators (in the sense of Rota), the ISP may be solved by describing the invariant subspaces of these operators alone. We characterize all anaytic Toeplitz operators $T_\phi$ on the Hardy space $H^2(\mathbb{D}^n)$ over the polydisk $\mathbb{D}^n$ for $n>1$ whose adjoints satisfy the \emph{Caradus criterion} for universality, that is, when $T_\phi^*$ is surjective and has infinite dimensional kernel. In particular if $\phi$ in a non-constant inner function on $\mathbb{D}^n$, or a polynomial in the ring $\mathbb{C}[z_1,\ldots,z_n]$ that has zeros in $\mathbb{D}^n$ but is zero-free on $\mathbb{T}^n$, then $T_\phi^*$ is universal for $H^2(\mathbb{D}^n)$. The analogs of these results for $n=1$ are not true.
	\end{abstract}
	
	\maketitle{}
	
	\section*{Introduction}
	One of the most important open problems in operator theory is the \emph{Invariant Subspace Problem} (ISP), which asks: Given a complex separable Hilbert space $\mathcal{H}$ and a bounded linear operator $T$ on $\mathcal{H}$, does $T$ have a nontrivial invariant subspace? An invariant subspace of $T$ is a closed subspace $E\subset\mathcal{H}$ such that $TE\subset E$. The recent monograph by Chalendar and Partington \cite{Chalendar Partington book} is a reference for some modern approaches to the ISP. In 1960, Rota \cite{Rota} demonstrated the existence of operators that have an invariant subspace structure so rich that they could model \emph{every} Hilbert space operator. \\ \\
	$\mathbf{Definition.}$
	\emph{Let $\mathcal{B}$ be a Banach space and $U$ a bounded linear operator on $\mathcal{B}$. Then $U$ is said to be universal for $\mathcal{B}$, if for any bounded linear operator $T$ on $\mathcal{B}$ there exists a constant $\alpha\neq 0$ and an invariant subspace $\mathcal{M}$ for $U$ such that the restriction $U|_\mathcal{M}$ is similar to $\alpha T$. } \\ 
	
	If $U$ is universal for a separable, infinite dimensional Hilbert space $\mathcal{H}$, then the ISP is equivalent to the assertion that every \emph{minimal} invariant subspace for $U$ is one dimensional. The main tool thus far  for identifying universal operators has been the following criterion of Caradus \cite{Caradus}. \\ \\
	$\mathbf{The \ Caradus \  Criterion.}$
	\emph{	Let $\mathcal{H}$ be a separable infinite dimensional Hilbert space and $U$ a bounded linear operator on $\mathcal{H}$. If $\mathrm{ker}(U)$ is infinite dimensional and $U$ is surjective, then $U$ is universal for $\mathcal{H}$. } \\
	
	Let $\mathbb{T}^n$ be the Cartesian product of $n$ copies of $\mathbb{T}=\partial\mathbb{D}$ equiped with the normalized Haar measure $\sigma$. Let $L^2(\mathbb{T}^n)$ denote the usual Lebesgue space and $L^\infty(\mathbb{T}^n)$ the essentially bounded functions with respect to $\sigma$. The Hardy space $H^2(\mathbb{D}^n)$ is the Hilbert space of holomorphic functions $f$ on $\mathbb{D}^n$ satisfying
	 \[
	 ||f||^2:=\sup_{0<r<1}\int_{\mathbb{T}^n}|f(r\zeta)|^2d\sigma(\zeta)<\infty.
	 \]
	 Denote by $H^\infty(\mathbb{D}^n)$ the space of bounded analytic functions on $\mathbb{D}^n$. It is well-known that both $H^2(\mathbb{D}^n)$ and $H^\infty(\mathbb{D}^n)$ can be viewed as subspaces of $L^2(\mathbb{T}^n)$ and $L^\infty(\mathbb{T}^n)$ respectively by identifying $f$ with its boundary function $f(\zeta):=\lim_{r\to 1}f(r\zeta)$ for almost every $\zeta\in\mathbb{T}^n$. If $|f|=1$ almost everywhere on $\mathbb{T}^n$ then $f$ is called \emph{inner}. Let $P$ be the orthogonal projection of $L^2(\mathbb{T}^n)$ onto $H^2(\mathbb{D}^n)$. The \emph{Toeplitz operator} $T_\phi$ with symbol $\phi\in L^\infty(\mathbb{T}^n)$ is defined by
	 \[
	 T_\phi f=P(\phi f)
	 \]
	for $f\in H^2(\mathbb{D}^n)$. Then $T_\phi$ is a bounded linear operator on $H^2(\mathbb{D}^n)$. The adjoint of $T_\phi$ is simply given by $T_\phi^*=T_{\bar{\phi}}$. If $\phi\in H^\infty(\mathbb{D}^n)$, then $T_\phi f=\phi f$ for all $f\in H^2(\mathbb{D}^n)$ and $T_\phi$ is called an \emph{analytic Toeplitz operator}.
	
	The best known examples of universal operators are all adjoints of analytic Toeplitz operators on $H^2(\mathbb{D})$, or are equivalent to one of them. For example $T^*_\phi$ when $\phi$ is a singular inner function or infinite Blaschke product. Another well known example was discovered by Nordgren, Rosenthal and Wintrobe \cite{Nordgren-Rosenthal-Wintrobe} in the 1980's. They proved that if $\phi$ is a hyperbolic automorphism of the unit disk $\mathbb{D}$,  then $C_\phi-\lambda I$ is universal for all $\lambda$ in the interior of the spectrum of $C_\phi$. In the last few years, Cowen and Gallardo-Gutiérrez (\cite{Cowen- Gallardo 1},\cite{Cowen- Gallardo 2},\cite{Cowen- Gallardo 3},\cite{Cowen- Gallardo 4}) have undertaken a thorough analysis of adjoints of analytic Toeplitz operators that are universal for $H^2(\mathbb{D})$.
	
	The objective of this article is to consider analytic Toeplitz operators $T_\phi$ whose adjoints are universal on $H^2(\mathbb{D}^n)$ for $n>1$. We now state our main result. 
	
	\begin{theorem*} \emph{Let $\phi\in H^\infty(\mathbb{D}^n)$ for $n>1$. Then $T_\phi^*$ satisfies the Caradus criterion for universality if and only if $\phi$ is invertible in $L^\infty(\mathbb{T}^n)$ but non-invertible in $H^\infty(\mathbb{D}^n)$}.
		\end{theorem*}

	This shows that analytic Toeplitz operators with universal adjoints are far more ubiquitous in the higher dimensional case $n>1$. \\ \\
	$\mathbf{Corollary.}$
	\emph{If $\phi$ is a non-constant inner function or a polynomial in the ring $\mathbb{C}[z_1,\ldots,z_n]$ that has zeros in $\mathbb{D}^n$ but no zeros in $\mathbb{T}^n$, then $T_\phi^*$ is universal for $H^2(\mathbb{D}^n)$ with $n>1$. } \\
	
	In particular the backward shift operators $T^*_{z_1},\ldots, T^*_{z_n}$ are universal when $n>1$. These results are \emph{not} true in general for $H^2(\mathbb{D})$. A closed subspace $E\subset H^2(\mathbb{D}^n)$ is said to be $z_i$-invariant if $z_i E\subset E$ for some $i=1,\ldots,n$. A $T_\phi$-invariant subspace $E$ is called maximal if $E\subsetneq H^2(\mathbb{D}^n)$ and there is no $T_\phi$-invariant subspace $L$ such that $E\subsetneq L \subsetneq H^2(\mathbb{D}^n)$. Since every maximal $T_\phi$-invariant subspace $E$  corresponds to a minimal $T^*_\phi$-invariant subspace $E^\perp$, we obtain the following version of the ISP. \\ \\
	$\mathbf{An \ equivalent \ version \ of \ the \ ISP.}$
	\emph{ Suppose $n>1$ and i=1,\ldots,n.	Then the ISP has a positive solution if and only if every maximal $z_i$-invariant subspace has codimension $1$ in $H^2(\mathbb{D}^n)$. } \\
	
	The maximal $z$-invariant subspaces in $H^2(\mathbb{D})$ indeed have codimension one by Beurling's Theorem. Hedenmalm \cite{Hedenmalm Maximal} proved that this is also true  in the classical Bergman space $L^2_a(\mathbb{D})$. A closed subspace $E\subset H^2(\mathbb{D}^n)$ that is $z_i$-invariant for all  $i=1,\ldots, n$ simultaneously is called a \emph{shift-invariant} subspace. The description of all shift-invariant subspaces for $n>1$ is a central open problem of multivariable operator theory (see the recent surveys of Sarkar \cite{Sarkar} and Yang \cite{Yang}{\tiny }) which could therefore lead to a solution of the ISP.
	\section{Proof of theorem}
	Cowen and Gallardo-Gutiérrez (\cite{Cowen- Gallardo 1},\cite{Cowen- Gallardo 2},\cite{Cowen- Gallardo 3},\cite{Cowen- Gallardo 4}) have frequently observed that if $\phi\in H^\infty(\mathbb{D})$ and there is an $\ell>0$ so that $|\phi(e^{i\theta}))|>\ell$ almost everywhere on $\mathbb{T}$, then $T_{1/\phi}$ is a left inverse for $T_\phi$ on $H^2(\mathbb{D})$ and therefore that $T_\phi^*$ is surjective. Our first goal is to characterize the left-invertibility of analytic Toeplitz operators on $H^2(\mathbb{D}^n)$ for all $n\geq 1$. The following general result can be found in functional analysis textbooks, but we state it here for completeness.
	
	\begin{lem}\label{Left invertible} Let $T$ be a bounded operator on $\mathcal{H}$. Then the following are equivalent.
		\begin{enumerate}
			\item $T$ is left-invertible.
		\item $T^*$ is surjective.
		\item $T$ is injective and has closed range.
			\end{enumerate}
		\end{lem}
	\begin{proof} If $T$ is left-invertible then there exists an bounded operator $S$ with $ST=I$ and hence $T^*S^*=I$. So $T^*$ is surjective and $(1)\implies (2)$. The equivalence $(2)\iff (3)$ is the closed range theorem \cite[Theorem 4.1.5]{Rudin}.
	If $T$ is injective and has closed range, then it has a bounded inverse $T^{-1}:T(\mathcal{H})\to\mathcal{H}$ by the open mapping theorem \cite[Corollary 2.12]{Rudin}. Therefore $T^{-1}T=I$ and $(3)\implies (1)$.
		\end{proof}
	Recently Koca and Sadik \cite{Koca-Sadik} proved that if $f\in H^\infty(\mathbb{D}^n)$, then the subspace $E=f  H^2(\mathbb{D}^n)$ is shift-invariant if and only if $f$ is invertible in $L^\infty(\mathbb{T}^n)$. Using this result we obtain the following characterization of left-invertibility for analytic $T_\phi$.
	
	\begin{prop}\label{Surjective} Let $T_\phi$ be an analytic Toeplitz operator on $H^2(\mathbb{D}^n)$ for $n\geq 1$. Then $T_\phi$ is left-invertible if and only if
	$\phi$ is invertible in $L^\infty(\mathbb{T}^n)$.	\end{prop}

\begin{proof} If $\phi$ is invertible in $L^\infty(\mathbb{T}^n)$, then $T_{1/\phi}$ is a (not necessarily analytic) Toeplitz operator and 
	\[
	T_{1/\phi}T_\phi f=T_{1/\phi}(\phi f)=Pf=f.
	\]
 for all $f\in H^2(\mathbb{D}^n)$. Therefore $T_\phi$ is left-invertible. Conversely if $T_\phi$ is left-invertible, then it has closed range by Lemma \ref{Left invertible}. Then $E=\phi H^2(\mathbb{D}^n)$ is closed and satisfies $z_i E\subset E$ for all $i=1,\ldots,n$. Hence $E$ is a shift-invariant subspace and therefore $\phi$ is invertible in $L^\infty(\mathbb{T}^n)$ by the result of Koca and Sadik \cite[Theorem 2]{Koca-Sadik}.
\end{proof}
 
 Therefore by Lemma \ref{Left invertible} and Proposition \ref{Surjective} we know precisely when the adjoint of an analytic Toeplitz operator is surjective. The key ingredient in the proof of the main theorem is a result of Ahern and Clark \cite[Theorem 3]{Aherk-Clark}. \\ \\
 $\mathbf{The \ Ahern \ and \ Clark \ Theorem.}$
 \emph{	If $f_1,\ldots,f_k\in H^2(\mathbb{D}^n) $ with $k<n$, then the shift-invariant subspace E generated by $f_1,\ldots,f_k$ is either all of $H^2(\mathbb{D}^n)$ or $E^\perp$ is infinite dimensional.} \\
 
 By the shift-invariant subspace $E$ \emph{generated} by $f_1,\ldots,f_k$ means the smallest shift-invariant subspace containing $f_1,\ldots,f_k$. So the  shift-invariant subspaces of the form $\phi H^2(\mathbb{D}^n)$ for some invertible $\phi\in L^\infty(\mathbb{T}^n)$ are singly generated. We arrive at our main result. \\ \\
  $\mathbf{Theorem.}$
 \emph{Let $\phi\in H^\infty(\mathbb{D}^n)$ for $n>1$. Then $T_\phi^*$ is surjective and has infinite dimensional kernel if and only if $\phi$ is invertible in $L^\infty(\mathbb{T}^n)$ but not in $H^\infty(\mathbb{D}^n)$.} 
 \begin{proof} Suppose $T_\phi^*$ is surjective and $\mathrm{Ker}(T_\phi^*)$ is infinite dimesnional. Then $\phi$ is invertible in $L^\infty(\mathbb{T}^n)$ by Lemma \ref{Left invertible} and Proposition \ref{Surjective}. It follows that the closed subspace $E=\phi H^2(\mathbb{D}^n)$ has infinite codimension since $E^\perp=\mathrm{Ker}(T_\phi^*)$. In particular $\phi H^2(\mathbb{D}^n)\neq H^2(\mathbb{D}^n)$ and hence $1/\phi\notin H^\infty(\mathbb{D}^n)$. Conversely suppose $\phi$ is invertible in $L^\infty(\mathbb{T}^n)$ but not in $H^\infty(\mathbb{D}^n)$. Then $T_\phi^*$ is surjective by Lemma \ref{Left invertible} and Proposition \ref{Surjective}. The assumption that  $1/\phi\notin H^\infty(\mathbb{D}^n)$ means that  $E=\phi H^2(\mathbb{D}^n)$ is a proper subspace of $H^2(\mathbb{D}^n)$. Therefore $E^\perp=\mathrm{Ker}(T_\phi^*)$ is infinite dimensional by the Ahern and Clark Theorem.
 	\end{proof}
 We immediately obtain the following dichotomy. \\ \\
$\mathbf{Corollary.}$
\emph{Let $T_\phi$ be a left-invertible analytic Toeplitz operator on $H^2(\mathbb{D}^n)$ for some $n>1$. Then either $T_\phi$ is invertible or $T_\phi^*$ is universal. } \\

		\section*{Acknowledgement}
		This work constitutes a part of the doctoral thesis of the first author, which is supervised by the second author.
	\bibliographystyle{amsplain}

\end{document}